\documentclass[8pt,onecolumn]{IEEEtran}
\usepackage{amssymb}
\usepackage{amsthm}
\usepackage{epsfig}
\usepackage{graphicx}
\usepackage{latexsym}
\usepackage{color}
\usepackage{enumerate}
\usepackage{changes}
\usepackage{mathrsfs}
\usepackage{epstopdf}
\usepackage{algorithm}
\usepackage{empheq}
\usepackage{algorithmic}
\usepackage{amsmath}
\newtheorem{proposition}{Proposition}
\newtheorem{theorem}{Theorem}
\newtheorem{assumption}{Assumption}
\newtheorem{lemma}{Lemma}
\newtheorem{remark}{Remark}

\begin{document}

\title{Distributed Optimization Algorithm with Superlinear
	Convergence Rate}

\author{
	Yeming Xu,~\IEEEmembership{}
	Ziyuan Guo,\IEEEmembership{}
	Kaihong Lu,\IEEEmembership {}
	Huanshui Zhang$^*$\IEEEmembership{}
	\thanks{H. Zhang is with the College of Electrical Engineering and Automation, Shandong University of Science and Technology, Qingdao 266590, China, and  also with the School of Control Science and Engineering, Shandong University, Jinan, Shandong 250061, China (e-mail: hszhang@sdu.edu.cn).}
		\thanks{ Y. Xu, Z. Guo, and K. Lu are with the School
		of Electrical and Automation Engineering, Shandong University of Science and Technology, Qingdao 266590, China (e-mail: ymxu2022@163.com; skdgzy@sdust.edu.cn; khong\_lu@163.com).}
	\thanks{This work was supported by the Original Exploratory Program Project of the National Natural Science Foundation of China (62450004),  the Joint Funds of the National Natural Science Foundation of China (U23A20325),  and the Major Basic Research of Natural Science Foundation of Shandong Province (ZR2021ZD14). (Corresponding author: Huanshui Zhang).}
\thanks{This work has been submitted to the IEEE for possible publication. Copyright may be transferred without notice, after which this version may no longer be accessible.}
}
\maketitle

\begin{abstract}
This paper considers distributed optimization problems, where each agent cooperatively minimizes the sum of local objective functions through the communication with its neighbors. 
The widely adopted distributed gradient method in solving this problem suffers from slow convergence rates, which motivates us to incorporate the second-order information of the objective functions. 
However, the challenge arises from the unique structure of the inverse of the Hessian matrix, which prevents the direct distributed implementation of the second-order method.
We overcome this challenge by proposing a novel optimization framework. The key idea is to transform the distributed optimization problem into an optimal control problem.  Using Pontryagin's maximum principle and the associated forward-backward difference equations (FBDEs), we derive a new distributed optimization algorithm that incorporates the second-order information without requiring the computation of the inverse of the Hessian matrix.
 Furthermore, the superlinear convergence of the proposed algorithm is proved under some mild assumptions.
 Finally, we also propose a variant of the algorithm to balance the number of iterations and communication.
\end{abstract}

\section{INTRODUCTION}
Consider that $n$ agents in a connected network collaboratively solve a consensus optimization problem
\begin{equation}
   \begin{aligned}
		\min \limits_{y \in \mathbb{R}^{p}}f(y)\triangleq \sum\limits_{i = 1}^n {{f_i}(y)},\label{original problem}\\
	\end{aligned}
\end{equation}
where $f: \mathbb{R}^{p} \mapsto \mathbb{R}$ is the objective function, while each agent $i$ privately holds a local objective function $f_i: \mathbb{R}^p \mapsto \mathbb{R}$ and can only communicate with its neighbors.
Problem \eqref{original problem} has received extensive attention in recent years due to its wide applications in many areas such as sensor networks \cite{khan2009diland,rabbat2004distributed,schizas2007consensus}, distributed control \cite{mota2014distributed,cao2012overview,lopes2008diffusion}, and machine learning \cite{koppel2018decentralized,lee2020optimization,bedi2019asynchronous}, etc.


Second-order methods are attractive because of their faster convergence rates, making it meaningful to consider their distributed implementations. 
Unfortunately, since the inverse of the Hessian matrix is dense, it is rather difficult to apply second-order methods directly to handling distributed optimization problems.
In fact, some Newton-type methods have been proposed to address this issue and have achieved notable progress. The distributed Adaptive Newton (DAN) algorithm with a global quadratic convergence rate is proposed in \cite{zhang2022distributed}, but this algorithm requires running a finite-time consensus inner loop in each iteration. Moreover, in a master-slave network configuration, some distributed quasi-Newton methods were proposed in \cite{shamir2014communication,wang2018giant,soori2020dave}, but the master node in this setup uses all the information from the slave nodes.
The network Newton (NN) method approximates the Newton step by truncating the Taylor expansion of the inverse Hessian matrix \cite{mokhtari2016network}.
 Another method, known as distributed quasi-Newton (DQN), represents a different family of distributed Newton-like methods to solve problem \eqref{original problem} \cite{bajovic2017newton}. Using the penalty method, DQN and NN converge to a solution neighborhood. 
 To further improve the convergence rate and solution accuracy, second-order methods working in the primal-dual domain have been proposed \cite{mokhtari2016decentralized,eisen2019primal,mokhtari2016dqm}. 
 The Newton tracking algorithm updates the local variable using a local Newton direction modified by neighboring and historical information \cite{zhang2021newton}. The distributed Newton-Raphson method proposed in \cite{varagnolo2015newton} requires exchanging gradient information and Hessian matrix information to estimate the global Newton direction. Both the Newton tracking and Newton-Raphson algorithms converge linearly to the optimal solution. Although these methods successfully incorporate second-order information into distributed optimization, they primarily focus on directly modifying or approximating Newton step. 

In our recent work, we successfully applied optimal control to solve optimization problems and achieved a series of significant results \cite{xu2023optimization,zhang2023optimization,wang2024superlinear}. Unlike the traditional  Newton-type methods, this paper proposes a new second-order method based on optimal control theory. We first reformulate the problem using a penalty method to exploit its specific structure, and then transform the optimization problem into an optimal control problem where the update size of each iteration is the control input, and the control objective is to design the current control input to minimize the sum of the original objective function and the update size for the future time instant.
Using Pontryagin’s maximum principle, we derive a new optimization algorithm. The proposed algorithm uses only neighbor information and incorporates second-order information through a special optimization structure. It avoids using the inverse of the Hessian matrix, significantly reducing computational complexity. Furthermore, we prove the superlinear convergence of the algorithm and establish its relationship with traditional algorithms. To balance the number of iterations and communication, we also present a variant of the algorithm. It is worth noting that with appropriate parameter selection, this variant can also inherit the aforementioned advantages.

The remainder is arranged as follows. In Section \ref{s2}, the distributed optimization problem is converted into an optimal control problem. 
In Section \ref{s3}, the novel algorithms, along with the variant to balance the number of iterations and communication, are proposed. 
The convergence analysis of these proposed algorithms is in Section \ref{s4}.  Concluding remarks of Section \ref{s5} complete the paper.

\textbf{Notation:} We use $\left\| A  \right\|$ and $\left\| x  \right\|$ to denote the Euclidean norm  of matrix $A\in  {\mathbb{R} ^{n \times n}}$ and vector $x\in \mathbb{R}^n$, respectively; $\mathbb{R}^n$ is the set of $n$-dimensional real vectors; 
$\mathbb{S}^n_{++}$ denotes the set of real symmetric positive definite matrices;
We use $I$ and $0$ to denote the identity matrix and all-zero matrix of proper dimensions; 
For symmetric matrices $A$ and $B$, $A \succ B$ means that the matrix $A-B$ is positive semidefinite.
Denote $\nabla f({x})$ and $\nabla^2 f({x})$ as the gradient and the Hessian matrix of $f({x})$;
The eigenvalues $\lambda_i(A)$ of a matrix $A \in {\mathbb{R} ^{n \times n}}$ are indexed in an decreasing order with respect to their real parts, i.e., $\Re ({\lambda _1}(A)) \ge  \cdots  \ge \Re ({\lambda _n}(A))$.

\section{PROBLEM FORMULATION}\label{s2}
The communication network among nodes is modeled by an undirected graph 
$\mathcal{G} = (\mathcal{V}, \mathcal{E})$, where $\mathcal{V} = \{1, \dots, n\}$ denotes the set of nodes and $\mathcal{E} \subseteq \mathcal{V} \times \mathcal{V}$ is the set of edges. An edge $(i, j) \in \mathcal{E}$ indicates that nodes $i$ and $j$ can directly send messages to each other. The neighbors of node $i$ are denoted by $\mathcal{N}_i$. The weight matrix of  $\mathcal{G}$ is a symmetric, doubly stochastic $n \times n$ matrix $W$ with elements $w_{ij}\ge0$. 
The weight $w_{ij} = 0$ if and only if $j \notin \mathcal{N}_i$, and $w_{ii} = 1 - \sum\limits_{j \in \mathcal{N}_i} w_{ij} > 0$.

We first provide a penalty interpretation of problem \eqref{original problem} to further explain the decentralized gradient descent algorithm. Then, we introduce an innovative idea for handling this problem using optimal control theory.

As we can see from \cite{mokhtari2016network,bajovic2017newton}, auxiliary penalty function $F:\mathbb{R}^{np}\mapsto \mathbb{R}$ of problem \eqref{original problem} can be written as 
\begin{align}
	\mathop {\min }\limits_{x \in \mathbb{R}^{np}} F(x) \triangleq h(x)+ \frac{1}{{2\lambda }}{x^T}(I - Z)x ,\label{P2}
\end{align}
where $x=(x^1,\dots,x^n)\in\mathbb{R}^{np}$, $ h(x)=\sum\limits_{i = 1}^n {{f_i}(x^i)}$, $x^i\in\mathbb{R}^p$ is the local decision vector of agent $i$, $\lambda$ is the penalty coefficient, and $Z=W\otimes I \in \mathbb{R}^{np\times np}$ is the Kronecker product of $W$ and $I$. A common numerical method to solve problem \eqref{P2} is using the gradient descent
\begin{align}
	 x_{k+1}=x_k-\lambda \nabla F(x_k),\label{GD}
\end{align}where $\nabla F(x_k)=\nabla h(x_k)+\frac{1}{\lambda }(I - Z){x_k}$ and $\nabla h(x_k)=(\nabla f_1(x_k^1),\dots,\nabla f_n(x_k^n))$. 
Indeed, the decentralized gradient descent can be considered as the component form of \eqref{GD}, i.e., the $i$th component of \eqref{GD} is
\begin{align}
	x_{k+1}^i=\sum\limits_{j \in {\mathcal{N}_i} \cup \{ i\} } {{w_{ij}}x_{k}^j - \lambda \nabla {f_i}(x_{k}^j)} ,i=1,\dots,n.\label{traditionl method}
\end{align}
Although the smaller step size means the solution of problem \eqref{P2} is closer to the solution of problem \eqref{original problem},  it also leads to slower convergence.


Gradient descent \eqref{GD} can be framed as the update of the state of a discrete-time linear time-invariant system
\begin{equation} 
	{x_{k + 1}} = {x_k} + {u_k}, \label{ztfc} 
\end{equation} where $x_k\in\mathbb{R}^{np}$ and $u_k=-\lambda \nabla F(x_k)\in\mathbb{R}^{np}$ are the state and control, respectively. 
Inspired by this, we aim to propose a numerical method to solve the problem \eqref{P2} by finding a new control law ${u_k}$ to update the state $x_k$ within the framework of the optimal control problem, i.e.,
  \begin{equation}
  	\begin{array}{l}
  		\mathop {{\rm{min}}}\limits_u {\rm{ }}\sum\limits_{k = 0}^N {(F({x_k}) + \frac{1}{2}u_k^{\rm{T}}} Ru_k^{}) + F({x_{N + 1}}),\\
  		{\rm subject\ to}{\rm {\; \eqref{ztfc}}},
  	\end{array}\label{OCP}
  \end{equation}
  where the initial condition ${x_0}$ is given, $N$ is the time horizon. The terminal cost is $F({x_{N+1}})$ and $R\in\mathbb{S}^{np}_{++}$ is the control weighted matrix.

\begin{remark}
	It is apparent from \eqref{OCP} that we minimize the accumulation of ${F(x_k)}$ and ${u_k^ \mathrm {T} R{u_k}}$. The optimal solution $u_t$ is in charge of regulating $x_{t+1}$ to minimize $\sum_{k=t+1}^N[F(x_k)+\frac{1}{2}u_k^TRu_k]+F(x_{N+1})$.
This means that we aim to control $x_k$ to reach the minimum point of $F(x)$ as quickly as possible while avoiding excessive costs associated with large controls. The purpose of transforming problem \eqref{P2} into problem \eqref{OCP} is to theoretically derive the fastest and stable algorithm with the form of \eqref{ztfc}.
\end{remark}

\section{DISTRIBUTED OPTIMIZATION METHOD VIA OPTIMAL CONTROL}\label{s3}
In this section, we use optimal control theory to solve problem \eqref{OCP}. This inspires us to develop an optimization algorithm that can be implemented in a distributed manner. Additionally, a variant of the algorithm is also proposed to balance the number of iterations and communication.
\subsection{Algorithm Development}
Following from \cite{naidu2018optimal}, applying Pontryagin’s maximum principle to problem \eqref{OCP} to yield the following costate equations:
\begin{align}
	&p_{k-1} = p_k + \nabla F(x_k), \label{backward} \\
	&0 = R u_k + p_k, \label{equil}
\end{align} with the terminal condition $p_N=\nabla F(x_{N+1})$. 
From the FBDEs \eqref{ztfc}, \eqref{backward}, \eqref{equil}, we can derive the relationship between the optimal control law \( u_k \) and the optimal state trajectory \( x_k \):
	\begin{align}
		u_k=&-R^{-1}\sum_{t=k+1}^{N+1}\nabla F(x_t),  \label{conlaw}\\
		x_{k+1}=&x_k-R^{-1}\sum_{t=k+1}^{N+1}\nabla F(x_t).\label{iteration relation}
	\end{align}
Noted that \eqref{iteration relation} does not apply to solving the optimal state directly due to its noncausality.
However, the special nonlinear form of \eqref{iteration relation} can be implemented by solving for $(x_k, u_k, p_k)$ such that the FBDEs \eqref{ztfc}, \eqref{backward}, \eqref{equil} are satisfied with the aid of the method of successive approximations (MSA) \cite{xu2023optimization}.
A straightforward algorithm will be presented in this subsection to simplify the calculation.
We define the first-order Taylor expansion of \eqref{conlaw} at $x_{t-1}$ as
\begin{align}
g_k\triangleq-R^{-1}\sum_{t=k+1}^{N+1}\big[\nabla h(x_{t-1})+\nabla^2 h(x_{t-1})(x_t-x_{t-1})+\frac{1}{\lambda}(I-Z)x_t\big],\label{tl}
\end{align}where $\nabla^2 h(x)$ is the block diagonal matrix with the $i$th diagonal block $\nabla^2f_i(x^i)$. Then, an approximate recursion for \eqref{iteration relation} is provided in Lemma \ref{Lm2}.
\begin{lemma}\label{Lm2}
	Dynamics \eqref{iteration relation}  is approximated by the following dynamics
	\begin{align}
		x_{k+1}=&x_{k}+g_k, \label{rimpite}\\
		g_k=-&\big[R+\nabla^2 h(x_k)+\frac{1}{\lambda}(I-Z)\big]^{-1}\big[\nabla h(x_k)+\frac{1}{\lambda}(I-Z)x_k\nonumber\\
		&+\sum_{t=k+1}^{N}(\nabla h(x_{t})+\frac{1}{\lambda}(I-Z)x_{t}-(\nabla^2h(x_{t})+\frac{1}{\lambda}(I-Z))g_{t})\big], \\ 
		g_N=-&\big[R+\nabla^2 h(x_N)+\frac{1}{\lambda}(I-Z)\big]^{-1}\big[{\nabla h(x_N)+\frac{1}{\lambda}(I-Z)x_N}\big].\label{gk}
	\end{align}
\end{lemma}
\begin{proof}
	The conclusion is achieved by substituting the first-order Taylor expansion $	g_k$ of the optimal control $u_k$ into \eqref{ztfc}, refer to \cite{zhang2023optimization} for the details.
\end{proof}
\begin{remark}
	Let $N=k$, \eqref{rimpite}-\eqref{gk} can be recovered as the centralized Newton method when $R=0$.  
\end{remark}
From lemma \ref{Lm2}, we say that $x_k$ generated by algorithm \eqref{rimpite}-\eqref{gk} is almost along the optimal state trajectory \eqref{iteration relation}.  Note that \eqref{rimpite}-\eqref{gk} still exhibit noncausality.
 To further eliminate the remaining noncausality and develop a distributed optimization algorithm, we modify \eqref{rimpite}-\eqref{gk} in a similar way as \cite{wang2024superlinear}, which obtains the following distributed optimization algorithm based on optimal control (DOBOC):
	\begin{align}
	&x_{k+1}=x_{k}-\hat g_k(x_k), \label{itewithhatgkm}\\
	&\hat g_l(x_k)=\eta\big[\nabla h(x_k)+\frac{1}{\lambda}(I-Z)x_k\big]+\big[I-\eta(\nabla^2h(x_k)+\frac{1}{\lambda}(I-Z))\big] \hat g_{l-1}(x_k), \label{hatgkm}\\
	&\hat g_0(x_k)=\eta\big[\nabla h(x_k)+\frac{1}{\lambda}(I-Z)x_k\big],l=1,\dots,k,\label{hatgkm1}
\end{align}where $\eta>0$ is an adjustable scalar.
The key reason we replace $\big[R+\nabla^2 h(x_k)+\frac{1}{\lambda}(I-Z)\big]^{-1}$ with the scalar $\eta$ in  \eqref{rimpite}-\eqref{gk} is that $\big[R+\nabla^2 h(x_k)+\frac{1}{\lambda}(I-Z)\big]^{-1}$ is dense. By making this substitution, the second-order information $\nabla^2 h(x_k) + \frac{1}{\lambda}(I - Z)$ is used in the DOBOC algorithm, which is as sparse as the network. 
Relying on the optimization framework based on optimal control, the DOBOC algorithm can be implemented in a distributed manner. 
\begin{lemma}
	The distribute implementation of the DOBOC algorithm for agent $i$ is
	\begin{align}
		x_{k+1}^i&=x_k^i-\hat g_k^i,\\
		\hat g_{l}^i &= \eta \big[\nabla {f_i}(x_k^i) +\frac{1}{\lambda}(x_k^i-\sum\limits_{j \in {\mathcal{N}_i} \cup \{ i\} }w_{ij}x_k^j)\big] +(1-\frac{\eta}{\lambda})\hat{g}_{l-1}^i-\eta\nabla^2f_i(x_k^i)\hat{g}_{l-1}^i+\frac{\eta}{\lambda}\sum\limits_{j \in {\mathcal{N}_i} \cup \{ i\} }w_{ij}\hat{g}_{l-1}^j,\\
		\hat g^i_0 &=  \eta \big[\nabla {f_i}(x_k^i) +\frac{1}{\lambda}(x_k^i-\sum\limits_{j \in {\mathcal{N}_i} \cup \{ i\} }w_{ij}x_k^j)\big], l=1,\dots,k.
	\end{align}
\end{lemma}
\begin{proof}
	The proof is immediately completed by rewriting \eqref{itewithhatgkm}-\eqref{hatgkm1} as the component form.
\end{proof}
The following remark illustrates the relationship between the DOBOC and the traditional optimization method.
\begin{remark}
	Noting that \eqref{hatgkm}-\eqref{hatgkm1} can be further recursively written as 
	\begin{align}
		\hat{g_k}(x_k)=&[I-(I-\eta(\nabla^2h(x_k)+\frac{1}{\lambda}(I-Z)))^{k+1}]\big[\nabla^2 h(x_k)+\frac{1}{\lambda}(I-Z)\big]^{-1}\big[\nabla h(x_k)+\frac{1}{\lambda}(I-Z)x_k\big].\label{hsf}
	\end{align}
It is still necessary to compute the inverse of $\nabla^2 h(x_k) + \frac{1}{\lambda} (I - Z)$ in \eqref{hsf}, which prevents directly applying \eqref{hsf} for distributed optimization.
The decentralized gradient descent \eqref{traditionl method} is included in our approach as a special case if we set $\eta=\lambda$ and $k=0$. If $\eta$ is chosen appropriately, $\hat g_k(x_k)$ can converge to the Newton step.
\end{remark}

Now we summarize the DOBOC algorithm from the perspective of each agent $i$ for handling problem \eqref{original problem}, refer to Algorithm 1 for details.
This algorithm only requires communication with neighbors, as described in Steps 3 and 6. Step 4 initializes the loop by computing the initial $\hat g^i_0$. Steps 6 and 7, as manifestations of optimal control theory, require receiving $\hat g^j_t$ from neighbors to obtain the update size for Step 9. 
\begin{algorithm}[H]
	\renewcommand{\algorithmicrequire}{\textbf{Input:}}
	\renewcommand{\algorithmicensure}{\textbf{Output:}}
	\caption{DOBOC.}
	\label{alg1}
	\begin{algorithmic}[1]
		\STATE Initialization: Each node $i$ requires $\eta,\lambda>0$ and sets $x_0^i\in\mathbb{R}^p$.
		\FOR {$k=0,1,\ldots$}
		\STATE Each agent $i$ transmits $x_k^i$ to all neighbors $j\in \mathcal{N}_i$ and receives $x_k^j$ from all $j\in \mathcal{N}_i$.
		\STATE Each agent $i$ calculates $\hat g^i_0 =  \eta \big[\nabla {f_i}(x_k^i) +\frac{1}{\lambda}(x_k^i-\sum\limits_{j \in {\mathcal{N}_i} \cup \{ i\} }w_{ij}x_k^j)\big]$.
		\FOR {$l=0,\dots,k-1$}
		\STATE Each agent $i$ receives $\hat g^j_l$ with neighbors  $j\in \mathcal{N}_i$.
		\STATE Each agent $i$ updates $\hat g_{l+1}^i = \eta \big[\nabla {f_i}(x_k^i) +\frac{1}{\lambda}(x_k^i-\sum\limits_{j \in {\mathcal{N}_i} \cup \{ i\} }w_{ij}x_k^j)\big] +(1-\frac{\eta}{\lambda})\hat{g}_{l}^i-\eta\nabla^2f_i(x_k^i)\hat{g}_{l}^i+\frac{\eta}{\lambda}\sum\limits_{j \in {\mathcal{N}_i} \cup \{ i\} }w_{ij}\hat{g}_{l}^j$.
		\ENDFOR
		\STATE Each agent $i$ updates $x_{k+1}^i=x_k^i-\hat g_k^i$.
		\ENDFOR	
	\end{algorithmic}  
\end{algorithm}

\subsection{A Usable Variant of DOBOC}
During the $k$th iteration, the DOBOC algorithm requires agent $i$ to transmit $\hat{g}_t^i, t = 0, \dots, k-1$ to its neighbors.
To balance the number of iterations and communication, we propose a usable variant of the DOBOC, called the DOBOC-$K$ algorithm, where $K$ denotes the restricted number of communications among agents in each iteration of the DOBOC algorithm. We summerize DOBOC-$K$ in the following lemma.
\begin{lemma}
		The distribute implementation of the DOBOC-$K$ algorithm for agent $i$ is
	\begin{align}
		x_{k+1}^i&=x_k^i-\hat g_{K-1}^i,\\
		\hat g_{t}^i &= \eta \big[\nabla {f_i}(x_k^i) +\frac{1}{\lambda}(x_k^i-\sum\limits_{j \in {\mathcal{N}_i} \cup \{ i\} }w_{ij}x_k^j)\big] +(1-\frac{\eta}{\lambda})\hat{g}_{t-1}^i-\eta\nabla^2f_i(x_k^i)\hat{g}_{t-1}^i+\frac{\eta}{\lambda}\sum\limits_{j \in {\mathcal{N}_i} \cup \{ i\} }w_{ij}\hat{g}_{t-1}^j,\\
		\hat g^i_0 &=  \eta \big[\nabla {f_i}(x_k^i) +\frac{1}{\lambda}(x_k^i-\sum\limits_{j \in {\mathcal{N}_i} \cup \{ i\} }w_{ij}x_k^j)\big],t=0,\dots,K-1.
	\end{align}
\end{lemma}
If $K=1$, i.e., agent $i$ only needs to transmit $x^i_k$ to its neighbors at each iteration, the DOBOC-$K$ algorithm can degenerate into the distributed gradient descent algorithm. For $K=2,\dots,k$, agent $i$ needs to additionally transmit $\hat g_0^i,\dots, \hat g_{K-2}^i$ to its neighbors at each iteration. 
This means that we approximate the update size $\hat g_k^i$ using $\hat g_{K-1}^i$.
The details of the distributed implementation of the DOBOC-$K$ are summarized in Algorithm 2.
\begin{algorithm}[H]
	\renewcommand{\algorithmicrequire}{\textbf{Input:}}
	\renewcommand{\algorithmicensure}{\textbf{Output:}}
	\caption{DOBOC-$K$.}
	\label{alg2}
	\begin{algorithmic}[2]
		\STATE	Initialization: Each agent $i$ requires $\eta,\lambda>0$, $K$ and sets $x_0^i\in\mathbb{R}^p$.
		\FOR {$k=0,1,\ldots$}
		\STATE Each agent $i$ transmits $x_k^i$ to all neighbors $j\in \mathcal{N}_i$ and receives $x_k^j$ from all $j\in \mathcal{N}_i$.
		\STATE Each agent $i$ calculates $\hat g^i_0 =  \eta \big[\nabla {f_i}(x_k^i) +\frac{1}{\lambda}(x_k^i-\sum\limits_{j \in {\mathcal{N}_i} \cup \{ i\} }w_{ij}x_k^j)\big]$.
		\FOR {$t=0,\dots,K-2$}
		\STATE Each agent $i$ receives $\hat g^j_t$ with neighbors  $j\in \mathcal{N}_i$.
		\STATE Each agent $i$ updates $\hat g_{t+1}^i = \eta \big[\nabla {f_i}(x_k^i) +\frac{1}{\lambda}(x_k^i-\sum\limits_{j \in {\mathcal{N}_i} \cup \{ i\} }w_{ij}x_k^j)\big] +(1-\frac{\eta}{\lambda})\hat{g}_{t}^i-\eta\nabla^2f_i(x_k^i)\hat{g}_{t}^i+\frac{\eta}{\lambda}\sum\limits_{j \in {\mathcal{N}_i} \cup \{ i\} }w_{ij}\hat{g}_{t}^j$.
		\ENDFOR
		\STATE Each agent $i$ updates $x_{k+1}^i=x_k^i-\hat g_{K-1}^i$.
		\ENDFOR	
	\end{algorithmic}  
\end{algorithm}

\section{CONVERGENCE ANALYSIS}\label{s4}
In this section, we reveal the superlinear convergence rate of the DOBOC algorithm. If we limit the number of communications per iteration in the DOBOC algorithm, the convergence rate of the DOBOC algorithm will degrade to at least linear.
The following assumption in this subsection is useful for our conclusion.
\begin{assumption} \label{A1}
	The local objective functions $f_i, i = 1,\dots, n$, are twice continuously differentiable, and there exist constants $0 < m  \le M < \infty $ such that for any $x\in \mathbb{R}^p$, $ mI \preceq {\nabla ^2}{f_i}(x) \preceq MI$.
\end{assumption}

Assumption \ref{A1} is standard in the analysis of the second method, see \cite{boyd2004convex}. 
The importance of Assumption \ref{A1} is that it allows us to establish bounds on the eigenvalues of ${\nabla ^2}F(x)$, as described below.
\begin{proposition}[\cite{mokhtari2016network}]\label{pp1}
	Under Assumption \ref{A1}, the eigenvalues of $\nabla^2F(x)$ are uniformly bounded as	$mI\preceq \nabla^2F(x)\preceq aI$, where $a=M+\frac{2(1-w_{min})}{\lambda}$ and $w_{min}={\min _{i}}\ {w_{ii}}$.
\end{proposition}
According to Proposition \ref{pp1}, for any $x,y\in\mathbb{R}^{np}$, there holds
\begin{align}
	&F(y)\ge  F(x)+\nabla F(x)^T(y-x)+\frac{m}{2}(y-x)^T(y-x),  \label{mu}\\
	&F(y)\le  F(x)+\nabla F(x)^T(y-x)+\frac{a}{2}(y-x)^T(y-x). \label{L}
\end{align}
Let $x_*$ be the minimum point of $F(x)$. We first prove the superlinear convergence of the DOBOC algorithm.
\begin{theorem}
	Under Assumption \ref{A1}, if $\eta$ satisfies $0<\eta<\frac{2}{a}$ and the sequence $\{x_k\}$ generated by the DOBOC algorithm converges, then 
	\begin{align}
		||x_{k+1}-x_*||\le c^{k+1}||x_k-x_*||,\nonumber
	\end{align}where $c=||I-\eta\nabla^2 F(x_*)||<1$.
\end{theorem}
\begin{proof}
	Direct derivation shows
	\begin{align}
		\hat g_k(x_*)=&[I-(I-\eta\nabla^2 F(x_*))^{k+1}]\nabla^2 F(x_*)^{-1}\nabla F(x_*)=0, \nonumber\\
		\nabla\hat g_k(x_*)=&I-(I-\eta\nabla^2 F(x_*))^{k+1}. \nonumber
	\end{align}
	From \eqref{itewithhatgkm}, it follows that
	\begin{align}
		&x_{k+1}-x_*\notag\\
		=&x_k-x_*-\hat g_k(x_k)\notag\\
		\approx&x_k-x_*-[\hat g_k(x_*)+\nabla \hat g_k(x_*)(x_k-x_*)]\notag\\
		=&x_k-x_*-\nabla\hat g_k(x_*)(x_k-x_*)\notag\\
		=&(I-\nabla\hat g_k(x_*))(x_k-x_*)\notag\\
		=&(I-\eta\nabla^2 F(x_*))^{k+1}(x_k-x_*).\nonumber
	\end{align}
	Since  $ ||I-\eta\nabla^2 F(x_*)||<1$ when $\eta<\frac{2}{a}$, it follows that $||x_{k+1}-x_*|| \le||I-\eta\nabla^2 F(x_*)||^{k+1}||x_k-x_*||$, i.e., the DOBOC algorithm is superlinearly convergent. The proof is completed.
\end{proof}

Before analyzing the convergence of DOBOC-$K$, we first introduce the following customary assumption, which can be found in \cite{zhang2022distributed,mokhtari2016network,zhang2021newton}.
\begin{assumption}
	The Hessians $\nabla^2 f_i(x)$ are $L$-Lipschitz continuous.\label{A4}
\end{assumption}
The restriction in Assumption \ref{A4} guarantees that  ${\nabla ^2}F(x)$ is Lipschitz continuous with $L$, as described in the following lemma.
\begin{lemma}\label{Lm4}
	Consider the function $F(x)$ defined in \eqref{P2}. If Assumptions \ref{A1} and \ref{A4} hold, then for any $x,y\in \mathbb{R}^{np}$
	\begin{align}
		F(y)\le F(x)+\nabla F(x)^T(y-x)+\frac{1}{2}(y-x)^T\nabla^2F(x)(y-x)+\frac{L}{6}||y-x||^3.
	\end{align}
\end{lemma}
\begin{proof}
 Let $y=(y_1,\dots,y_n)$ and $x=({x_1},\dots,{x_n})$, then
	\begin{align}
		||\nabla^2F(y)-\nabla^2F(x)||=\mathop {\max }\limits_{{\rm{i = 1,}} \cdots {\rm{n}}}||\nabla^2f_i(y_i)-\nabla^2f_i(x_i)||\le L\mathop {\max }\limits_{{\rm{i}}} ||y_i-x_i||\le L||y-x||,
	\end{align}
The subsequent proof can be obtained directly using the Taylor expansion error as discussed in \cite{apostol1991calculus}.
\end{proof}
 The following lemma provides the bound for the error $F(x_{k+1})-F(x_*)$ in terms of $F(x_k)-F(x_*)$, which clarifies the proof of the convergence of the DOBOC-$K$ algorithm.
\begin{lemma}\label{Lm5}
	Under Assumptions \ref{A1} and \ref{A4}, if select the step size as
	\begin{align}
		\eta <\min \left\{ {1, \frac{1}{a},\frac{2m}{{{a^2}{{K }^2}}} }\right\},\label{f}
	\end{align}where $a=M+\frac{2(1-w_{min})}{\lambda}$, then by DOBOC-$K$ algorithm, there holds
	\begin{align}
	F(x_{k+1})-F(x_*)\le &\left[1- {\frac{2m^2\eta-ma^2\eta^2K^2}{a}} \right](F(x_k)-F(x_*)) +\frac{a^3(2a)^{\frac{3}{2}}\eta^3K^3L}{6m^3}(F(x_k)-F(x_*))^{\frac{3}{2}}.\label{ddgx}
\end{align}
\end{lemma}
\begin{proof}
The DOBOC-$K$ algorithm can be rewritten in a centralized form $x_{k+1}=x_k-h\nabla^2F(x_k)^{-1}\nabla F(x_k)$, where $h=I-(I-\eta \nabla^2F(x_k))^{K}$. Assuming $\eta < \frac{1}{a}$, it can be easily shown that
	\begin{align}
		\lambda_1(h)&\le 1-(1-\eta K\lambda_1(\nabla^2F(x_k)))=\eta K\ \lambda_1(\nabla^2F(x_k))\le a\eta K\nonumber\\
		\lambda_n(h)&\ge1-(1-\eta\lambda_n(\nabla^2F(x_k)))\ge \eta \ \lambda_n(\nabla^2F(x_k))\ge m\eta>0.\nonumber
	\end{align} by using the Bernoulli inequality  $(1-x)^k \ge 1 - kx$ for all $x \in (0,1)$ \cite{mitrinovic2013classical} .
	It can be verified from
	\begin{align}
		&h\nabla^2F(x)^{-1}\nonumber\\
		=&\nabla^2F(x)^{-1}-(I-\eta \nabla^2F(x))(I-\eta \nabla^2F(x))\cdots(I-\eta \nabla^2F(x))\nabla^2F(x)^{-1}\nonumber\\
		=&\nabla^2F(x)^{-1}-(I-\eta \nabla^2F(x))(I-\eta \nabla^2F(x))\cdots\nabla^2F(x)^{-1}(I-\eta \nabla^2F(x))\nonumber\\
		&\vdots\nonumber\\
		=&\nabla^2F(x)^{-1}-\nabla^2F(x)^{-1}(I-\eta \nabla^2F(x))(I-\eta \nabla^2F(x))\cdots(I-\eta \nabla^2F(x))\nonumber\\
		=&\nabla^2F(x)^{-T}h^T\nonumber
	\end{align}and $\lambda_n(h\nabla^2F(x)^{-1})>\lambda_n(h)\lambda_n(\nabla^2F(x)^{-1})>0$ that $h\nabla^2F(x)^{-1}\succ 0$. We can also prove that $\nabla^2F(x)h\succ0$ by using the similar method. Substituting $y$ and $x$ with $x_{k+1}$ and $x_k$ respectively in Lemma \ref{Lm4} yields
	\begin{align}
		F({x}_{k+1})\le& F(x_k)+\nabla F(x_k)^T({x_{k+1}}-x_k)+\frac{1}{2}({x_{k+1}}-x_k)^T\nabla^2F(x_k)({x_{k+1}}-x_k)+\frac{L}{6}||x_{k+1}-x_k||^3\nonumber\\
		=&F(x_k)-\nabla F(x_k)^Th\nabla^2F(x_k)^{-1}\nabla F(x_k)+\frac{1}{2}\nabla F(x_k)^T\nabla^2F(x_k)^{-T} h^T\nabla^2F(x_k)h\nabla^2F(x_k)^{-1}\nonumber\\
		&\times\nabla F(x_k)+\frac{L}{6}||h\nabla^2F(x_k)^{-1}\nabla F(x_k)||^3\nonumber\\
		\le&F(x_k)-m\eta\nabla F(x_k)^T\nabla^2F(x_k)^{-1}\nabla F(x_k)+\frac{a^2\eta^2K^2}{2}\nabla F(x_k)^T\nabla^2F(x_k)^{-1}\nabla F(x_k)\nonumber\\
		&+\frac{a^3\eta^3K^3L}{6}||\nabla^2F(x_k)^{-1}\nabla F(x_k)||^3\nonumber\\
		\le&F(x_k)-\frac{2m\eta-a^2\eta^2K^2}{2a}\nabla F(x_k)^T\nabla F(x_k)+\frac{a^3\eta^3K^3L}{6m^3}||\nabla F(x_k)||^3.\label{bds1}
	\end{align}
	From  \eqref{mu} and \eqref{L}, one has
	\begin{align}
		F(x_*)\ge F(x_k)-\frac{1}{2m}\nabla F(x_k)^T\nabla F(x_k), \label{key1}\\
		F(x_*)\le F(x_k)-\frac{1}{2a}\nabla F(x_k)^T\nabla F(x_k). \label{key2}
	\end{align}
Setting $\eta < \frac{2m}{\left[a(K-1)\right]^2}$ and substituting \eqref{key1} and \eqref{key2} into \eqref{bds1}, one has
	\begin{align}
		F(x_{k+1})-F(x_*)\le &\left[1- {\frac{2m^2\eta-ma^2\eta^2K^2}{a}} \right](F(x_k)-F(x_*)) +\frac{a^3(2a)^{\frac{3}{2}}\eta^3K^3L}{6m^3}(F(x_k)-F(x_*))^{\frac{3}{2}}.\nonumber
	\end{align} 
\end{proof}
The theorem below shows that the DOBOC-$K$ algorithm can converge to the optimal value $F(x_*)$ with at least linear convergence rate.
\begin{theorem}\label{Th1}
		Consider $x_{k}$ are generated by the DOBOC-$K$. If Assumptions \ref{A1} and \ref{A4} hold, and  the step size is chosen to satisfy \eqref{f} and
\begin{align}
		\eta <\min \left\{ {\frac{m}{{{a^2}{{K}^2}}},{{\left[ {\frac{{6{m^5}}}{{{a^4}{{(2a)}^{\frac{2}{3}}}{K}^3L(F({x_0}) - F({x_*}))^{\frac{1}{2}}}}} \right]}^{\frac{1}{2}}}} \right\},\nonumber
\end{align} then the sequence $F(x_k)$ converges to the optimal argument $F(x_*)$ at least linear as 
	\begin{align}
			F(x_{k})-F(x_*)\le (1-{\scriptsize \epsilon})^k(F(x_0)-F(x_*)),\nonumber
		\end{align}where the constant $0<\epsilon<1$ is given by
	\begin{align}
			 \epsilon=\frac{{2{m^2}\eta  - m{a^2}{}{\eta ^2}{K}^2}}{a} - \frac{a^3(2a)^{\frac{3}{2}}\eta^3K^3L}{6m^3}(F(x_0)-F(x_*))^{\frac{1}{2}}.\nonumber
		\end{align}
	\end{theorem}
\begin{proof}
    From lemma \ref{Lm5}, we can rewrite \eqref{ddgx} as
	\begin{align}
		F(x_{k+1})-F(x_*)\le(1-\beta_k)(F(x_k)-F(x_*)),\nonumber
	\end{align}where
\begin{align}
	\beta_k=\frac{{2{m^2}\eta  - m{a^2}{}{\eta ^2}{K}^2}}{a} - \frac{a^3(2a)^{\frac{3}{2}}\eta^3K^3L}{6m^3}(F(x_k)-F(x_*))^{\frac{1}{2}}.\label{beta}
\end{align}
It easily to find
	\begin{align}
		\beta_k\le \frac{2ma\eta K- m{a^2}{\eta ^2}{{K}^2}}{a}\le \frac{m}{a}<1.\nonumber
	\end{align}
	We now turn to proof $\beta_k>0$.  Notice that $2-\frac{a^2\eta K^2}{m}>1$ when $\eta<\frac{m}{a^2K^2}$, thus 
	\begin{align}
		\beta_0=&\frac{{2{m^2}\eta  - m{a^2}{}{\eta ^2}{K}^2}}{a} - \frac{a^3(2a)^{\frac{3}{2}}\eta^3K^3L}{6m^3}(F(x_0)-F(x_*))^{\frac{1}{2}}\nonumber\\
		\ge&\frac{m^2\eta}{a}- \frac{a^3(2a)^{\frac{3}{2}}\eta^3K^3L}{6m^3}(F(x_0)-F(x_*))^{\frac{1}{2}}.\nonumber
	\end{align}
	Set $\eta<{{\left[ {\frac{{6{m^5}}}{{{a^4}{{(2a)}^{\frac{2}{3}}}{K^3}L(F({x_0}) - F({x_*}))^{\frac{1}{2}}}}} \right]}^{\frac{1}{2}}}$, then $\beta_0>0$.  It immediately get $F({x_1}) - F({x_*})<F({x_0}) - F({x_*})$, which implys  $\beta_0<\beta_1$ from \eqref{beta}. Using the similar method, we can prove $0<\beta_k<\beta_{k+1}<1$ for all index $k$. Set $\epsilon=\beta_0$, the proof is completed.
\end{proof}

\section{CONCLUSION}\label{s5}
In this paper, we proposed a new algorithm called DOBOC to solve distributed optimization problems. The DOBOC algorithm is based on the penalty method interpretation of distributed optimization and optimal control theory. It successfully integrates second-order information into distributed optimization without requiring the inverse of the Hessian matrix. A connection has been established between DOBOC and traditional optimization methods. To balance the trade-off between the number of iterations and communication, we also proposed a variant of the DOBOC, named DOBOC-$K$. The results on the convergence analysis of DOBOC and DOBOC-$K$ have been achieved, and both algorithms have demonstrated superior performance.

\bibliographystyle{IEEEtran}
\bibliography{reference}

\end{document}